\documentclass[12pt, reqno]{amsart}
\usepackage{ amsmath,amsthm, amscd, amsfonts, amssymb, graphicx, color}
\usepackage[bookmarksnumbered, colorlinks, plainpages,linkcolor=blue,urlcolor=blue,citecolor=blue]{hyperref}
\usepackage{setspace}
\usepackage{ulem}

\newtheorem{thm}{Theorem}[section]
\newtheorem{cor}[thm]{Corollary}
\newtheorem{lem}[thm]{Lemma}

\newtheorem{exam}[thm]{Example}
\numberwithin{equation}{section}

\begin{document}

\title{Perturbation for group inverse in a Banach algebra}

\author{Dayong Liu}
\author{Tugce Pekacar Calci}
\author[H. Kose]{Handan Kose}
\author{Huanyin Chen}
\address{Dayong Liu, College of Science, Central South University of Forestry and Technology, Changsha, China}
\email{<liudy@csuft.edu.cn>}
\address{Tugce Pekacar Calci, Department of Mathematics, Ankara University, Ankara, Turkey}
\email{<tcalci@ankara.edu.tr>}
\address{Handan Kose, Department of Mathematics, Ahi Evran University, Kirsehir, Turkey}
\email{handan.kose@ahievran.edu.tr}
\address{Huanyin Chen, Department of Mathematics, Hangzhou Normal University, Hangzhou, China}
\email{<huanyinchen@aliyun.com>}

\subjclass[2010]{15A09, 47L10, 32A65.}
\keywords{group inverse; generalized Drazin inverse; additive property; perturbation; Banach algebra.}
\thanks{Corresponding author: Huanyin Chen, huanyinchen@aliyun.com}

\begin{abstract}
We present new additive results for the group inverse in a Banach algebra under certain perturbations. The upper bound of $\|(a+b)^{\#}-a^d\|$ is thereby given.
These extend the main results in [X. Liu, Y. Qin and H. Wei, Perturbation bound of the group inverse and the generalized Schur complement in Banach algebra, {\it Abstr. Appl. Anal.}, 2012, 22 pages. \url{https://doi.org/10.1155/2012/629178}].
\end{abstract}
\maketitle

\section{Introduction}

Let $\mathcal{A}$ be a complex Banach algebra with an identity $1$ and $a\in \mathcal{A}$. The Drazin inverse of $a$ is the unique $x$ in $\mathcal{A}$ satisfying
$$ ax=xa, \ \ xax=x, \ \ a^{k+1}x=a^k $$
for some nonnegative integers $k$. We denote $x$ by $a^{D}$. The least such integer $k$ is called the index of $a$, denoted by $ind(a)=k$. If $ind(a)=1$, we say $a$ is group invertible, that is, $ax=xa, xax=x, axa=a.$ The element $x$ will be denoted by $a^{\#}$. An element $x$ is a generalized Drazin inverse of $a$ in $\mathcal{A}$ if
$$ax=xa, \ \ xax=x, \ \ a-a^2x\in {\mathcal{A}}^\mathrm{{qnil}}.$$
The generalized Drazin inverse is unique if it exists. We denote $x$ by $a^d$. Here, ${\mathcal{A}}^\mathrm{{qnil}}$ is the sets of all quasinilpotent elements in $\mathcal{A}$, i.e., $a\in
\mathcal{A}^{qnil}\Leftrightarrow \lim\limits_{n\to \infty}\parallel
a^n\parallel^{\frac{1}{n}}=0.$

The preceding generalized inverses play important roles in matrix and operator theory (see~\cite{CM2018,CM2020,LQB,MZC,Zhang-Mosic,Z3}).
In ~\cite{Yang-Liu}, Yang and Liu studied the Drazin inverse of the sum of complex matrices $P$ and $Q$ under the condition $QP^2=0$ and $QPQ=0$.
The generalized Drazin inverse of $a+b$ in a Banach algebra was studied in~\cite[Theorem 2.1]{LQ} under the condition $ba^2=0$ and $b^2=0$.
In~\cite{L}, Liu et al. investigated the Drazin inverse $(P+Q)^D$ of two complex matrices $P$ and $Q$ under the conditions $P^2Q=PQP$ and $Q^2P=QPQ$.
In ~\cite{Z3}, Zou et al. presented the generalized Drazin inverse of $a+b$ in a Banach algebra under $a^2b=aba$ and $b^2a=bab$.

Liu et al. investigated the perturbation of group inverse under the conditions $\| a^dbaa^d\|<1, a^{\pi}ba=0$ (see~\cite[Theorem 2.8]{LQW}). They further studied the group invertibility under the condition $\| a^dbaa^d\|<1, a^{\pi}ba=aba^{\pi}$ (see~\cite[Theorem 2.12]{LQW}). These inspire us to explore the perturbation for group invertibility in a more general setting. The preceding results of X. Liu et al. are thereby generalzied to wider cases.

In Section 2, we investigate the group invertibility of $a+b$ in a Banach algebra under the condition
$a^{\pi}ba^2=0, a^{\pi}bab=0$. The upper bound of $\|(a+b)^{\#}-a^d\|$ is thereby given. In Section 3, we are concerned with the additive property for group inverse under the condition $a^2ba^{\pi}=a^{\pi}aba, a^{\pi}b^2a=baba^{\pi}$ and obtain the bound of $\|(a+b)^{\#}-a^d\|$.

Throughout the paper, all Banach algebras are complex with an identity. We use $\mathcal{A}^{d}$ to denote the set of all g-Drazin invertible elements in $\mathcal{A}$. We use $a^{\pi}$ to stand for the spectral idempotent $1-aa^d$ of an element $a\in \mathcal{A}^d$.

\section{orthogonal conditions}

In this section we establish additive property of group inverse in a Banach algebra under certain orthogonal perturbation conditions. We begin with

\begin{lem}\label{lem1}
 Let $\mathcal{A}$ be a Banach algebra, $a, b\in \mathcal{A}^{\mathrm{d}}$,
$$x=\left(\begin{array}{cc}
           a   &   0       \\
           c   &   b
\end{array}\right)_p
~\mbox{or}~
    \left(\begin{array}{cc}
           b   &   c       \\
           0   &   a
\end{array}
\right)_p .$$
Then
$$x^d=\left(\begin{array}{cc}
           a^d &   0       \\
           z   &   b^d
\end{array}\right)_p,
~\mbox{or}~
    \left(\begin{array}{cc}
           b^d &   z       \\
           0   &   a^d
\end{array}\right)_p,$$
where
$$\begin{array}{c}
z=(b^d)^2\big(\sum\limits_{i=0}^{\infty}(b^d)^ica^i\big)a^{\pi}+b^{\pi}\big(\sum\limits_{i=0}^{\infty}b^ic(a^d)^i\big)(a^d)^2-b^dca^d.
\end{array}$$
\end{lem}
\begin{proof}
See ~\cite[Lemma 1.1]{Zhang-Mosic}.
\end{proof}

\begin{lem}\label{lem2}
Let $e$ be an idempotent in a Banach algebra $\mathcal{A}$, and let $a, ea\in\mathcal{A}^{d}$. If $ea(1-e)=0$, then $a(1-e)\in \mathcal{A}^{d}$, and $$(ea)^d=ea^d, (a(1-e))^d=a^d(1-e).$$
\end{lem}
\begin{proof} See~\cite[Lemma 2.2]{Zhang-Mosic}.\end{proof}

We are ready to prove:

\begin{thm}\label{thm1}
Let $\mathcal{A}$ be a Banach algebra, and let $a,b,a^{\pi}b\in \mathcal{A}^{d}$. If $$\| a^dbaa^d\|<1, a^{\pi}ba^2=0, a^{\pi}bab=0,$$
then $a+b$ has group inverse if and only if $a^{\pi}(a+b)$ has group inverse. In this case,
$$\begin{array}{rl}
&\|(a+b)^{\#}-a^d\|\\
\leq& \displaystyle\left(\frac{\| a^d \|}{1-\| a^dbaa^d \|}\right)^2\| ba^{\pi}\|\displaystyle\bigg(\| a^{\pi}\|+\sum\limits_{n=0}^{\infty}\| a^{\pi}a\|^{n+1}\| a^{\pi}b^d\|^{n+1}\\
+&\displaystyle\sum\limits_{n=0}^{\infty}\| a^{\pi}a\|^{n+2}\| a^{\pi}b^d\|^{n+2}\displaystyle+\| a^{\pi}b^d\|(\| a^{\pi}a\|+\| a^{\pi}b\|){\bigg)}  \\
+&\displaystyle\big(1+\frac{\|a^db\|}{1-\|a^dbaa^d\|}\big)\left(\| a^d\|+\sum\limits_{n=0}^{\infty}\| a^{\pi}a\|^n\| a^{\pi}b^d\|^{n+1}\right.      \\
+&\displaystyle\left.\sum\limits_{n=0}^{\infty}\| a^{\pi}a\|^{n+1}\| a^{\pi}b^d\|^{n+2}\right).
\end{array}$$
\end{thm}
\begin{proof}
Let $p=aa^{d}$. Then we have
$$a=\left(\begin{array}{cc}
a_1&0\\
0&a_2
\end{array}
\right)_p, b=\left(\begin{array}{cc}
b_{1}&b_{2}\\
b_{3}&b_4
\end{array}
\right)_p,$$
where $a_1\in (p{\mathcal{A}}p)^{-1}$, $a_2\in ((1-p){\mathcal{A}}(1-p))^{\mathrm{{qnil}}}$.

Furthermore,
$$a^d=\left(\begin{array}{cc}
a_1^{-1}&0\\
0&0
\end{array}
\right)_p~\mbox{and}~a^{\pi}
=\left(\begin{array}{cc}
0  &  0         \\
0  &  a^{\pi}
\end{array}
\right)_p.$$
Since $a^{\pi}ba^2=0$, we see that $a^{\pi}baa^d=a^{\pi}ba^2(a^d)^2=0$. This implies that $b_3=0$, and so
$$a+b=\left(\begin{array}{cc}
a_1+b_1&b_2\\
0&a_2+b_4
\end{array}
\right)_p.$$
Since $\| a^dbaa^d\|<1$, we see that $\| a_1^{-1}b_1\|<1$, and so $a_1+b_1=a_1(1+a_1^{-1}b_1)\in \big(p\mathcal{A}p)^{-1}$.
One easily checks that $a_2=a-a^2a^d=aa^{\pi}\in ((1-p)\mathcal{A}(1-p))^{\mathrm{qnil}}$.
On the other hand, we have $b_4=a^{\pi}ba^{\pi}=a^{\pi}b-a^{\pi}ba^2(a^d)^2=a^{\pi}b\in \mathcal{A}^{d}$, and so $b_4\in \big((1-p)\mathcal{A}(1-p)\big)^d$.
Moreover, we have
$$\begin{array}{c}
b_4a_2^2=a^{\pi}b(1-aa^d)a^2a^{\pi}=a^{\pi}ba^2a^{\pi}=0,\\
b_4a_2b_4=a^{\pi}b(1-aa^d)aba^{\pi}=a^{\pi}baba^{\pi}=0.
\end{array}$$
Noting that $a_2^d=0$, by virtue of \cite[Theorem 2.1]{Yang-Liu}, we have
$$\begin{array}{rcl}
(a_2+b_4)^d   &   =   &    \sum\limits_{n=0}^{\infty}a_2^n(b_4^d)^{n+1}+\sum\limits_{n=0}^{\infty}a_2^n(b_4^d)^{n+2}a_2       \\
                 &   =   &    \sum\limits_{n=0}^{\infty}(a^{\pi}a)^n(a^{\pi}b^d)^{n+1}+\sum\limits_{n=0}^{\infty}(a^{\pi}a)^n(a^{\pi}b^d)^{n+2}a^{\pi}a .
\end{array}$$
Therefore, by \cite[Theorem 2.3]{CK}, $$(a+b)^d=\left(\begin{array}{cc}
(a_1+b_1)^{-1}&z\\
0&(a_2+b_4)^d
\end{array}
\right)_p,$$
where $$\displaystyle z=\sum_{n=0}^{\infty}(a_1+b_1)^{-n-2}b_2(a_2+b_4)^n(a_2+b_4)^{\pi}-(a_1+b_1)^{-1}b_2(a_2+b_4)^d.$$
We check that $$a_2+b_4=a^{\pi}a+a^{\pi}b=a^{\pi}(a+b).$$
If $a+b\in \mathcal{A}^{\#}$, we easily see that $a_2+b_4\in \mathcal{A}^{\#}$, and so $a^{\pi}(a+b) \in {\mathcal{A}}^{\#}$.
We now assume that $a^{\pi}(a+b)\in \mathcal{A}^{\#}.$ Then
$(a_2+b_4)^d=(a_2+b_4)^{\#}$.
Obviously, $$\begin{array}{c}
(a+b)^d(a+b)=(a+b)(a+b)^d,\\
(a+b)^d=(a+b)^d(a+b)(a+b)^d.
\end{array}$$
It is easy to verify that
$$\renewcommand{\arraystretch}{1.5}
\begin{array}{l}
\big((a_1+b_1)z+b_2(a_2+b_4)^{\#}\big)(a_2+b_4)\\
=\big((a_1+b_1)^{-1}b_2(a_2+b_4)^{\pi}-b_2(a_2+b_4)^{\#}\big)(a_2+b_4)\\
\indent+b_2(a_2+b_4)^{\#}(a_2+b_4)\\
=0.
\end{array}$$
Then we have
$$\begin{array}{l}
\vspace{2mm}
(a+b)(a+b)^d(a+b)\\
\vspace{2mm}
=\left(\begin{array}{cc}
a_1+b_1&b_2\\
0&a_2+b_4
\end{array}
\right)
\left(\begin{array}{cc}
(a_1+b_1)^{-1}&z\\
0&(a_2+b_4)^d
\end{array}\right)\\
\vspace{2mm}
\indent\left(\begin{array}{cc}
a_1+b_1&b_2\\
0&a_2+b_4
\end{array}\right)\\
=a+b.
\end{array}$$
Therefore $(a+b)^{\#}=(a+b)^d$, i.e., $a+b$ has group inverse. That is,
$$\begin{array}{rcl}
	(a_2+b_4)^{\#}  &  =  &   \sum\limits_{n=0}^{\infty}(a^{\pi}a)^n(a^{\pi}b^d)^{n+1}+\sum\limits_{n=0}^{\infty}(a^{\pi}a)^n(a^{\pi}b^d)^{n+2}a^{\pi}a .
\end{array}$$
Since $\| a_1^{-1}b_1\|<1$, we easily check that
$$\begin{array}{rcl}
(a_1+b_1)^{-1}   &   =   &    (1+a_1^{-1}b_1)^{-1}a_1^{-1}      \\
                 &   =   &    \sum\limits_{n=0}^{\infty}(-1)^n(a_1^{-1}b_1)^n a_1^{-1}       \\
                &   =   &    \sum\limits_{n=0}^{\infty}(-1)^n(a^dbaa^d)^n a^d  \\
                                &   =   &    \sum\limits_{n=0}^{\infty}(-1)^n(a^db)^n a^d.
\end{array}$$
On the other hand,
$$\begin{array}{rcl}
(a_1+b_1)^{-1}   &   =   &    a_1^{-1}(1+b_1a_1^{-1})^{-1}      \\
                 &   =   &    a_1^{-1}\sum\limits_{n=0}^{\infty}(-1)^n(ba_1^{-1})^n      \\
                 &   =   &    a^d\sum\limits_{n=0}^{\infty}(-1)^n(ba^d)^n.
\end{array}$$
Also we have\\
$\renewcommand{\arraystretch}{2.5}
\begin{array}{rcl}
z    &   =   &    (a_1+b_1)^{-2}b_2(a_2+b_4)^{\pi}-(a_1+b_1)^{-1}b_2(a_2+b_2)^{\#}       \\
     &   =   &    \left(\sum\limits_{n=0}^{\infty}(-1)^n(a^db)^na^d\right)^2\cdot aa^d ba^{\pi}                   \\
     &       &    \cdot\left(a^{\pi}-a^{\pi}(a+b)\left(\sum\limits_{n=0}^{\infty}(a^{\pi}a)^n(a^{\pi}b^d)^{n+1}
                  +\sum\limits_{n=0}^{\infty}(a^{\pi}a)^n(a^{\pi}b^d)^{n+2}a^{\pi}a\right)\right)   \\
     &       &    -\left(\sum\limits_{n=0}^{\infty}(-1)^n(a^db)^na^d\right)\cdot aa^d ba^{\pi}     \\
     &       &    \cdot\left(\sum\limits_{n=0}^{\infty}(a^{\pi}a)^n(a^{\pi}b^d)^{n+1}
                  +\sum\limits_{n=0}^{\infty}(a^{\pi}a)^n(a^{\pi}b^d)^{n+2}a^{\pi}a\right)             \\
     &   =   &    \left(\sum\limits_{n=0}^{\infty}(-1)^n(a^db)^na^d\right)^2  ba^{\pi}\left(a^{\pi}-\sum\limits_{n=0}^{\infty}(a^{\pi}a)^{n+1}(a^{\pi}b^d)^{n+1}\right.                  \\
     &       &    \left.-\sum\limits_{n=0}^{\infty}(a^{\pi}a)^{n+1}(a^{\pi}b^d)^{n+2}a^{\pi}a-a^{\pi}b^da^{\pi}(a+b)\right)   \\
     &       &    -\left(\sum\limits_{n=0}^{\infty}(-1)^n(a^db)^{n+1}\right)       \\
     &       &    \cdot \left(\sum\limits_{n=0}^{\infty}(a^{\pi}a)^{n}(a^{\pi}b^d)^{n+1}
                        +\sum\limits_{n=0}^{\infty}(a^{\pi}a)^{n}(a^{\pi}b^d)^{n+2}a^{\pi}a\right).
\end{array}$
Hence
$$(a+b)^{\#}=(a_1+b_1)^{-1}+z+(a_2+b_4)^{\#}.$$
We compute that
$$(a+b)^{\#}-a^d= \sum\limits_{n=1}^{\infty}(-1)^n(a^db)^na^d$$
$\renewcommand{\arraystretch}{2.5}
\begin{array}{rcl}
&       &    +\left(\sum\limits_{n=0}^{\infty}(-1)^n(a^db)^na^d\right)^2  ba^{\pi}\left(a^{\pi}-\sum\limits_{n=0}^{\infty}(a^{\pi}a)^{n+1}(a^{\pi}b^d)^{n+1}\right.                  \\
&       &    \left.-\sum\limits_{n=0}^{\infty}(a^{\pi}a)^{n+1}(a^{\pi}b^d)^{n+2}a^{\pi}a-a^{\pi}b^da^{\pi}(a+b)\right)   \\
\end{array}$  \\
\indent
$\renewcommand{\arraystretch}{2.5}
\begin{array}{rcl}
&       &    -\left(\sum\limits_{n=0}^{\infty}(-1)^n(a^db)^{n+1}\right)\left(\sum\limits_{n=0}^{\infty}(a^{\pi}a)^{n}(a^{\pi}b^d)^{n+1}
+\sum\limits_{n=0}^{\infty}(a^{\pi}a)^{n}(a^{\pi}b^d)^{n+2}a^{\pi}a\right)                        \\
&       &    +\sum\limits_{n=0}^{\infty}(a^{\pi}a)^{n}(a^{\pi}b^d)^{n+1}
+\sum\limits_{n=0}^{\infty}(a^{\pi}a)^{n}(a^{\pi}b^d)^{n+2}a^{\pi}a                        \\
\end{array}$  \\
$\renewcommand{\arraystretch}{2.5}
\begin{array}{rcl}
	&   =   &    \left(\sum\limits_{n=0}^{\infty}(-1)^n(a^db)^na^d\right)^2  ba^{\pi}\left(a^{\pi}-\sum\limits_{n=0}^{\infty}(a^{\pi}a)^{n+1}(a^{\pi}b^d)^{n+1}\right.                  \\
	&       &    \left.-\sum\limits_{n=0}^{\infty}(a^{\pi}a)^{n+1}(a^{\pi}b^d)^{n+2}a^{\pi}a-a^{\pi}b^da^{\pi}(a+b)\right)   \\
	&       &    +\sum\limits_{n=0}^{\infty}(-1)^n(a^db)^n\left[a^d+\sum\limits_{n=0}^{\infty}(a^{\pi}a)^{n}(a^{\pi}b^d)^{n+1}
	+\sum\limits_{n=0}^{\infty}(a^{\pi}a)^{n}(a^{\pi}b^d)^{n+2}a^{\pi}a\right].
\end{array}$ \\
\vskip2mm
\noindent
Clearly, we have
$$\begin{array}{c}
\|\sum\limits_{n=0}^{\infty}(-1)^n(a^db)^na^d\|\leq \sum\limits_{n=1}^{\infty}\| a^dbaa^d\|^{n-1}\|a^d\|=\frac{\|a^d\|}{1-\|a^dbaa^d\|},\\
\|\sum\limits_{n=0}^{\infty}(-1)^n(a^db)^n\|\leq 1+\|\sum\limits_{n=0}^{\infty}(-1)^n(a^dbaa^d)^n\| \|a^db\|=1+\frac{\|a^db\|}{1-\|a^dbaa^d\|}.
\end{array}$$
Therefore  \\

$\|(a+b)^{\#}-a^d\|$  \\

\hfill
$\begin{array}{l}
\leq\displaystyle\left(\frac{\| a^d \|}{1-\| a^dbaa^d \|}\right)^2\| ba^{\pi}\|\displaystyle\bigg(\| a^{\pi}\|+\sum\limits_{n=0}^{\infty}\| a^{\pi}a\|^{n+1}\| a^{\pi}b^d\|^{n+1}\\
+\displaystyle\sum\limits_{n=0}^{\infty}\| a^{\pi}a\|^{n+2}\| a^{\pi}b^d\|^{n+2}\displaystyle+\| a^{\pi}b^d\|(\| a^{\pi}a\|+\| a^{\pi}b\|){\bigg)}  \\
+\displaystyle\big(1+\frac{\|a^db\|}{1-\|a^dbaa^d\|}\big)\left(\| a^d\|+\sum\limits_{n=0}^{\infty}\| a^{\pi}a\|^n\| a^{\pi}b^d\|^{n+1}\right.
\end{array}$  \\

\noindent
$\begin{array}{l}
\hspace{13.4mm}+\displaystyle\left.\sum\limits_{n=0}^{\infty}\| a^{\pi}a\|^{n+1}\| a^{\pi}b^d\|^{n+2}\right).
\end{array}$
\end{proof}

\begin{cor}\label{cor1}
Let $\mathcal{A}$ be a Banach algebra, and let $a,b,a^{\pi}b\in \mathcal{A}^{d}$. If
$$\begin{array}{c}
a^{\pi}ba^2=0, a^{\pi}bab=0,\\
max\{\| a^dbaa^d\|, \| a^{\pi}a\|\| a^{\pi}b^d\|\}<1,
\end{array}$$
then $a+b$ has group inverse if and only if $a^{\pi}(a+b)$ has group inverse. In this case,
$$\renewcommand{\arraystretch}{2.5}
\begin{array}{rcl}
\|(a+b)^{\#}-a^d\|
	&   \leq   &   \displaystyle\left(\frac{\| a^d \|}{1-\| a^dbaa^d\|}\right)^2\| ba^{\pi}\|\left(\| a^{\pi}\| +\frac{\| a^{\pi}a\|\| a^{\pi}b^d\|}{1-\| a^{\pi}a\|\| a^{\pi}b^d\|}\right.     \\
	&          &   \displaystyle\left.+\frac{\left(\| a^{\pi}a\|\| a^{\pi}b^d\|\right)^2}{1-\| a^{\pi}a\|\| a^{\pi}b^d\|}+\| a^{\pi}b^d\|\left(\| a^{\pi}a\|+\| a^{\pi}b\|\right)\right)  \\
	&          &   \displaystyle+\big(1+\frac{\|a^db\|}{1-\|a^dbaa^d\|}\big)\left(\| a^d\|+\frac{\| a^{\pi}b^d\|}{1-\| a^{\pi}a\|\| a^{\pi}b^d\|} \right.     \\
	&          &   \displaystyle\left.+\frac{\| a^{\pi}a\|\| a^{\pi}b^d\|^2}{1-\| a^{\pi}a\|\| a^{\pi}b^d\|}\right).      \\
\end{array}$$
\end{cor}
\begin{proof} By hypothesis, $\| a^{\pi}a\|\| a^{\pi}b^d\|<1$. Therefore we complete the proof of Theorem \ref{thm1}.
\end{proof}

The following example illustrates that Theorem 2.3 is a nontrivial generalization of ~\cite[Theorem 2.8]{LQW}.

\begin{exam} Let
$$a=\begin{pmatrix}
           0    &   1    &    0      \\
           0    &   0    &    1      \\
           0    &   0    &    0      \\
        \end{pmatrix},
    b=\begin{pmatrix}
           0    &   -1    &    0      \\
           0    &   0    &    -1      \\
           0    &   0    &    0      \\
        \end{pmatrix}\in {\mathbb{C}^{3\times 3}}. $$
Clearly, $a^d=b^d=0$ and $a^{\pi}=1$. Therefore we have $$\| aa^dba^d\|=0<1, a^{\pi}ba^2=0 ~\mbox{and} ~a^{\pi}bab=0.$$ But $$a^{\pi}ba=
\left(
\begin{array}{ccc}
0&0&-1\\
0&0&0\\
0&0&0
\end{array}
\right)\neq 0.$$
\end{exam}

\section{commutative conditions}

In this section we shall establish the representation of group inverse of $a+b$ under a kind of commutative perturbation condition which is weaker than $a^{\pi}ba=aba^{\pi}$. The following lemma is crucial.

\begin{lem}\label{lem3}
Let $\mathcal{A}$ be a Banach algebra, and let $a\in \mathcal{A}^{qnil}, b\in \mathcal{A}$. If $a^2b=aba$, $b^2a=bab$, then $a+b\in {\mathcal{A}}^d$. In this case,
$$(a+b)^d=\sum_{n=0}^{\infty}(b^d)^{n+1}(-a)^n+b^{\pi}a\sum_{n=0}^{\infty}(-1)^n(n+1)(b^d)^{n+2}a^n.$$
\end{lem}
\begin{proof}
Since $a\in \mathcal{A}^{qnil}$, we have $a^d=0$. This completes the proof by ~\cite[Theorem 3.3]{Z3}.	
\end{proof}

We are ready to prove:

\begin{thm}\label{thm2}
Let $\mathcal{A}$ be a Banach algebra, and let $a,b,a^{\pi}b\in \mathcal{A}^{d}$. If $$\| a^dbaa^d\|<1, a^2ba^{\pi}=a^{\pi}aba, a^{\pi}b^2a=baba^{\pi},$$
then $a+b$ has group inverse if and only if $(a+b)a^{\pi}$ has group inverse. In this case,  \\

\noindent
$\|(a+b)^{\#}-a^d\| $  \\
\indent
\hfill
$\renewcommand{\arraystretch}{2.5}
\begin{array}{rcl}
&\leq & \displaystyle\frac{\|a^db\|\|a^d\|}{1-\|a^db\|}+\|a^{\pi}baa^d\|\left(\frac{\|a^d\|}{1-\|a^db\|}\right)^2\\
&&\displaystyle +\|(a+b)a^{\pi}baa^d\|\left(\frac{\|a^d\|}{1-\|a^db\|}\right)^2\sum_{n=0}^{\infty}\|b^da^{\pi}\|^{n+1}\|aa^{\pi}\|^n\\
&&\displaystyle +\|b^{\pi}aa^{\pi}\|\|(a+b)a^{\pi}baa^d\|\sum_{n=0}^{\infty}(n+1)\|b^da^{\pi}\|^{n+2}\|aa^{\pi}\|^n     \\
&&\displaystyle \cdot\left(\frac{\|a^d\|}{1-\|a^db\|}\right)^2+\sum_{n=0}^{\infty}\|b^da^{\pi}\|^{n+1}\|aa^{\pi}\|^n
\end{array}$  \\

\indent
\hfill
$\renewcommand{\arraystretch}{2.5}
\begin{array}{rcl}
 &       &   \displaystyle +\|a^{\pi}baa^d\|\frac{\|a^d\|}{1-\|a^db\|}\sum_{n=0}^{\infty}\|b^da^{\pi}\|^{n+1}\|aa^{\pi}\|^n    \\
 &       &   \displaystyle +\|b^{\pi}aa^{\pi}\|\sum_{n=0}^{\infty}(n+1)\|b^da^{\pi}\|^{n+2}\|aa^{\pi}\|^n    \\
 &       &   \displaystyle +\|b^{\pi}aa^{\pi}\|\|a^{\pi}baa^d\|\frac{\|a^d\|}{1-\|a^db\|}\sum_{n=0}^{\infty}(n+1)\|b^da^{\pi}\|^{n+2}\|aa^{\pi}\|^n.
\end{array}$
\end{thm}
\begin{proof}
Let $p=aa^{d}$. Then we have
$$a=\left(\begin{array}{cc}
a_1  &  0     \\
0    &  a_2
\end{array}
\right)_p,
b=\left(\begin{array}{cc}
b_{1}    &    b_{3}    \\
b_{4}    &    b_{2}
\end{array}
\right)_p,$$
where $a_1$ is invertible, $a_2$ is quasinilpotent. Then
$$a^d=\left(\begin{array}{cc}
a_1^{-1}&0\\
0&0
\end{array}
\right)_p~\mbox{and}~a^{\pi}
=\left(\begin{array}{cc}
0    &    0          \\
0    &    a^{\pi}
\end{array}
\right)_p.$$
Since $a^2ba^{\pi}=a^{\pi}aba$, we check that
$$b_3=aa^dba^{\pi}=(a^d)^2a^2ba^{\pi}=(a^d)^2a^{\pi}aba=0.$$
Therefore,
$$a+b=\left(\begin{array}{cc}
	a_1+b_1     &     0          \\
	b_4         &     a_2+b_2
	\end{array}\right).$$
Obviously, $a^db=a^dbaa^d=a_1^{-1}b_1$. Since $\| a^dbaa^d\|<1$, we have $\| a_1^{-1}b_1 \|<1$. Thus $a_1+b_1=a_1(1+a_1^{-1}b_1)$ is invertible in $p{\mathcal{A}}p$. Moreover, $a_1+b_1$ is group invertible in $p{\mathcal{A}}p$. In view of $a^2ba^{\pi}=a^{\pi}aba$, we get $b_2=a^{\pi}ba^{\pi}=ba^{\pi}-(a^d)^2a^2ba^{\pi}=ba^{\pi}-(a^d)^2a^{\pi}aba=ba^{\pi}$. That is, $a_2+b_2=(a+b)a^{\pi}$.  Therefore, by \cite[Theorem 2.1]{M} and the invertibility of $a_1+b_1$, we derive that $a+b$ is group invertible if and only if $(a+b)a^{\pi}$ is group invertible. In this case,   \\

\noindent
$(a+b)^{\#}$  \\
\hfill $=\left(\begin{array}{cc}
(a_1+b_1)^{-1}                                                         &        0                   \\
(a_2+b_2)^{\pi}b_4(a_1+b_1)^{-2}-(a_2+b_2)^{\#}b_4(a_1+b_1)^{-1}       &        (a_2+b_2)^{\#}
\end{array}\right).$  \\

\noindent
As $a_2=aa^{\pi}\in \mathcal{A}^{\mathrm{qnil}}$ and $b_2=a^{\pi}ba^{\pi}$, it follows from $a^2ba^{\pi}=a^{\pi}aba$ that $$aa^{\pi}\cdot aa^{\pi}\cdot a^{\pi}ba^{\pi}=aa^{\pi}\cdot a^{\pi}ba^{\pi}\cdot aa^{\pi}.$$
That is, $a_2^2b_2=a_2b_2a_2$. Likewise, $b_2^2a_2=b_2a_2b_2$. By hypothesis, $a^{\pi}b\in \mathcal{A}^d$, by using Cline's formula,
$b_2\in \mathcal{A}^d$. In light of Lemma \ref{lem3}, $a_2+b_2\in ((1-p)\mathcal{A}(1-p))^d$. Therefore,
$$\begin{array}{l}
(a_2+b_2)^{\#}=(a_2+b_2)^d  \\
\hspace{7mm}=\displaystyle\sum_{n=0}^{\infty}(-1)^n(b_2^d)^{n+1}a_2^n+b_2^{\pi}a_2\sum_{n=0}^{\infty}(-1)^n(n+1)(b_2^d)^{n+2}a_2^n   \\
\hspace{7mm}=\displaystyle\sum_{n=0}^{\infty}(-1)^n(b^da^{\pi})^{n+1}(aa^{\pi})^n     \\
  \hspace{30mm}\displaystyle+b^{\pi}aa^{\pi}\sum_{n=0}^{\infty}(-1)^n(n+1)(b^da^{\pi})^{n+2}(aa^{\pi})^n.
  \end{array}$$
We derive that  \\
\noindent

$(a+b)^{\#} $  \\
\noindent
$\begin{array}{rcl}
 &    =    &    (a_1+b_1)^{-1}+(a_2+b_2)^{\pi}b_4(a_1+b_1)^{-2}             \\
 &         &    -(a_2+b_2)^{\#}b_4(a_1+b_1)^{-1}+(a_2+b_2)^{\#}
 \end{array}$  \\
\noindent
$\renewcommand{\arraystretch}{2.5}
\begin{array}{rcl}
 &    =    &    \displaystyle\sum_{n=0}^{\infty}(-1)^n(a^db)^na^d+\left[a^{\pi}
                            -\left(\sum_{n=0}^{\infty}(-1)^n(b^da^{\pi})^{n+1}(aa^{\pi})^n\right.\right.   \\
 \end{array}$  \\
\noindent
$\renewcommand{\arraystretch}{2.5}
\begin{array}{rcl}
 &         &    \displaystyle\left.\left.+b^{\pi}aa^{\pi}\sum_{n=0}^{\infty}(-1)^n(n+1)(b^da^{\pi})^{n+2}(aa^{\pi})^n\right)(a+b)a^{\pi}\right]     \\
 &         &    \displaystyle\cdot a^{\pi}baa^d \left(\sum_{n=0}^{\infty}(-1)^n(a^db)^na^d\right)^2
                            +\left(\sum_{n=0}^{\infty}(-1)^n(b^da^{\pi})^{n+1}(aa^{\pi})^n\right.
 \end{array}$  \\
\noindent
$\renewcommand{\arraystretch}{2.5}
\begin{array}{rcl}
 &         &    \displaystyle \left.+b^{\pi}aa^{\pi}\sum_{n=0}^{\infty}(-1)^n(n+1)(b^da^{\pi})^{n+2}(aa^{\pi})^n\right)
 \end{array}$  \\
\noindent
$\renewcommand{\arraystretch}{2.5}
\begin{array}{rcl}
 &         &    \displaystyle\cdot \left(a^{\pi}-a^{\pi}baa^d\sum_{n=0}^{\infty}(-1)^n(a^db)^na^d \right)
\end{array}$  \\
\noindent
$\renewcommand{\arraystretch}{2.5}
\begin{array}{rcl}
 &   =   &   \displaystyle\sum_{n=0}^{\infty}(-1)^n(a^db)^na^d+a^{\pi}baa^d\left(\sum_{n=0}^{\infty}(-1)^n(a^db)^na^d\right)^2
\end{array}$  \\
\noindent
$\renewcommand{\arraystretch}{2.5}
\begin{array}{rcl}
 &       &   \displaystyle-\sum_{n=0}^{\infty}(-1)^n(b^da^{\pi})^{n+1}(aa^{\pi})^n(a+b)a^{\pi}baa^d\left(\sum_{n=0}^{\infty}(-1)^n(a^db)^na^d\right)^2
\end{array}$  \\
\noindent
$\renewcommand{\arraystretch}{2.5}
\begin{array}{rcl}
 &       &   \displaystyle -b^{\pi}aa^{\pi}\sum_{n=0}^{\infty}(-1)^n(n+1)(b^da^{\pi})^{n+2}(aa^{\pi})^n(a+b)a^{\pi}baa^d     \\
 &       &   \displaystyle \cdot\left(\sum_{n=0}^{\infty}(-1)^n(a^db)^na^d\right)^2+\sum_{n=0}^{\infty}(-1)^n(b^da^{\pi})^{n+1}(aa^{\pi})^n
\end{array}$  \\
\noindent
$\renewcommand{\arraystretch}{2.5}
\begin{array}{rcl}
 &       &   \displaystyle -\sum_{n=0}^{\infty}(-1)^n(b^da^{\pi})^{n+1}(aa^{\pi})^na^{\pi}baa^d\sum_{n=0}^{\infty}(-1)^n(a^db)^na^d    \\
 &       &   \displaystyle +b^{\pi}aa^{\pi}\sum_{n=0}^{\infty}(-1)^n(n+1)(b^da^{\pi})^{n+2}(aa^{\pi})^n    \\
 &       &   \displaystyle -b^{\pi}aa^{\pi}\sum_{n=0}^{\infty}(-1)^n (n+1)(b^da^{\pi})^{n+2}(aa^{\pi})^na^{\pi}baa^d\sum_{n=0}^{\infty}(-1)^n(a^db)^na^d.
\end{array}$   \\

\noindent
From this, we obtain the estimation of norm     \\

$\|(a+b)^{\#}-a^d\|\leq\displaystyle\frac{\|a^db\|\|a^d\|}{1-\|a^db\|}+\|a^{\pi}baa^d\|\left(\frac{\|a^d\|}{1-\|a^db\|}\right)^2$ \\

\hfill
$\renewcommand{\arraystretch}{2.2}
\begin{array}{rcl}
  &       &   \displaystyle +\|(a+b)a^{\pi}baa^d\|\left(\frac{\|a^d\|}{1-\|a^db\|}\right)^2\sum_{n=0}^{\infty}\|b^da^{\pi}\|^{n+1}\|aa^{\pi}\|^n    \\
 &       &   \displaystyle +\|b^{\pi}aa^{\pi}\|\|(a+b)a^{\pi}baa^d\|\sum_{n=0}^{\infty}(n+1)\|b^da^{\pi}\|^{n+2}\|aa^{\pi}\|^n     \\
 &       &   \displaystyle \cdot\left(\frac{\|a^d\|}{1-\|a^db\|}\right)^2+\sum_{n=0}^{\infty}\|b^da^{\pi}\|^{n+1}\|aa^{\pi}\|^n      \\
 &       &   \displaystyle +\|a^{\pi}baa^d\|\frac{\|a^d\|}{1-\|a^db\|}\sum_{n=0}^{\infty}\|b^da^{\pi}\|^{n+1}\|aa^{\pi}\|^n    \\

\end{array}$  \\

$\hspace{3mm}\displaystyle +\|b^{\pi}aa^{\pi}\|\sum_{n=0}^{\infty}(n+1)\|b^da^{\pi}\|^{n+2}\|aa^{\pi}\|^n.$

$\hfill\displaystyle+\|b^{\pi}aa^{\pi}\|\|a^{\pi}baa^d\|\frac{\|a^d\|}{1-\|a^db\|}\sum_{n=0}^{\infty}(n+1)\|b^da^{\pi}\|^{n+2}\|aa^{\pi}\|^n.$
\end{proof}

As an immediate consequence of Theorem \ref{thm2}, we now derive
\begin{cor}\label{cor2}
Let $\mathcal{A}$ be a Banach algebra, and let $a,b,a^{\pi}b\in \mathcal{A}^{d}$. If
\begin{center}
$a^2ba^{\pi}=a^{\pi}aba$, $a^{\pi}b^2a=baba^{\pi}$,\\
$max\{\| a^dbaa^d\|, \| aa^{\pi}\|\| b^da^{\pi}\|\}<1$,
\end{center}
then $a+b$ has group inverse if and only if $(a+b)a^{\pi}$ has group inverse. In this case,     \\

\vspace{-3mm}
$\|(a+b)^{\#}-a^d\| $  \\

\vspace{-3mm}
\hfill
$\renewcommand{\arraystretch}{2.4}
\begin{array}{rcl}
 & \leq  &   \displaystyle\frac{\|a^db\|\|a^d\|}{1-\|a^db\|}+\|a^{\pi}baa^d\|\left(\frac{\|a^d\|}{1-\|a^db\|}\right)^2    \\
 &       &   \displaystyle +\|(a+b)a^{\pi}baa^d\|\left(\frac{\|a^d\|}{1-\|a^db\|}\right)^2\frac{\|b^da^{\pi}\|}{1-\|b^da^{\pi}\|\|aa^{\pi}\|}    \\
 &       &   \displaystyle +\|b^{\pi}aa^{\pi}\|\|(a+b)a^{\pi}baa^d\|\left(\frac{\|b^da^{\pi}\|}{1-\|b^da^{\pi}\|\|aa^{\pi}\|}\right)^2     \\
 &       &   \displaystyle \cdot\left(\frac{\|a^d\|}{1-\|a^db\|}\right)^2+\frac{\|b^da^{\pi}\|}{1-\|b^da^{\pi}\|\|aa^{\pi}\|}      \\
 &       &   \displaystyle +\|a^{\pi}baa^d\|\frac{\|a^d\|}{1-\|a^db\|}\frac{\|b^da^{\pi}\|}{1-\|b^da^{\pi}\|\|aa^{\pi}\|}    \\
 &       &   \displaystyle +\|b^{\pi}aa^{\pi}\|\left(\frac{\|b^da^{\pi}\|}{1-\|b^da^{\pi}\|\|aa^{\pi}\|}\right)^2    \\
 &       &   \displaystyle +\|b^{\pi}aa^{\pi}\|\|a^{\pi}baa^d\|\frac{\|a^d\|}{1-\|a^db\|}\left(\frac{\|b^da^{\pi}\|}{1-\|b^da^{\pi}\|\|aa^{\pi}\|}\right)^2
\end{array}$
\end{cor}
\begin{proof}
The first part follows from Theorem \ref{thm2}.  For the second part, by $\| aa^{\pi}\|\| b^da^{\pi}\|<1$, it is easy to compute that
$$\sum_{n=0}^{\infty}\|b^da^{\pi}\|^{n+1}\|aa^{\pi}\|^n=\frac{\|b^da^{\pi}\|}{1-\|b^da^{\pi}\|\|aa^{\pi}\|},$$
$$\sum_{n=0}^{\infty}(n+1)\|b^da^{\pi}\|^{n+2}\|aa^{\pi}\|^n=\left(\frac{\|b^da^{\pi}\|}{1-\|b^da^{\pi}\|\|aa^{\pi}\|}\right)^2.$$
So the upper bound of the norm $\|(a+b)^{\#}-a^d\| $ can be obtained by Theorem \ref{thm2}.
\end{proof}

\begin{cor} (see~\cite[Theorem 2.12]{LQW}) \label{thm3}
Let $\mathcal{A}$ be a Banach algebra, and let $a,b\in \mathcal{A}^{d}$. If $$\| a^dbaa^d\|<1, a^{\pi}ba=aba^{\pi},$$
then $a+b$ has group inverse if and only if $a^{\pi}(a+b)$ has group inverse. In this case,
$$\|(a+b)^{\#}-a^d\|\leq  \displaystyle\frac{\|a^db\|\|a^d\|}{1-\|a^db\|}+\|a^{\pi}\|\sum\limits_{n=0}^{\infty}\|b^d\|^{n+1}\|a\|^n.$$
\end{cor}
\begin{proof} Since $a^{\pi}ba=aba^{\pi}$, we have $aa^dba^{\pi}=a^d(a^{\pi}ba)=0$. Hence $aa^db=aa^dbaa^d$.
    Also we have $a^{\pi}baa^d=aba^{\pi}a^d=0$, and so $baa^d=aa^dbaa^d$. Accordingly,
    $a^dab=baa^d$. Hence $a^{\pi}b=ba^{\pi}$. Therefore we check that
$$\begin{array}{c}
a^2ba^{\pi}=a(aba^{\pi})=a(a^{\pi}ba)=a^{\pi}aba,\\
a^{\pi}b^2a=a^{\pi}ba^{\pi}ba=a^{\pi}baba^{\pi}=a^{\pi}bab=ba^{\pi}ab=baba^{\pi}.
\end{array}$$ Since $a^{\pi}ba=aba^{\pi}$, we have $a^{\pi}baa^d=aba^{\pi}a^d=0$.
By virtue of Theorem \ref{thm2}, we obtain the upper bound of the norm  \\
\indent
\hspace{8mm}$\|(a+b)^{\#}-a^d\|  \leq \displaystyle\frac{\|a^db\|\|a^d\|}{1-\|a^db\|}+\|a^{\pi}\|\sum\limits_{n=0}^{\infty}\|b^d\|^{n+1}\|a\|^n. $
\end{proof}

We present a numerical example to demonstrate Theorem \ref{thm2} as follow.

\begin{exam} Let
$a=\left(\begin{array}{cccc}
1   &   0   &   0   &   0   \\
0   &   1   &   0   &   0   \\
0   &   0   &   0   &   1   \\
0   &   0   &   0   &   0
\end{array}\right),
b=\left(\begin{array}{cccc}
\frac{1}{2}  &  0  &  0  &  0                   \\
0            &  0  &  0  &  0                   \\
0            &  0  &  0  &  0                   \\
0            &  0  &  0  &  2
\end{array}\right)\in \mathbb{C}^{4\times 4}.$
Then
$$a^d=\left(\begin{array}{cccc}
	1   &   0   &   0   &   0   \\
	0   &   1   &   0   &   0   \\
	0   &   0   &   0   &   0   \\
	0   &   0   &   0   &   0
\end{array}\right),
a^{\pi}=\left(\begin{array}{cccc}
	0   &   0   &   0   &   0   \\
	0   &   0   &   0   &   0   \\
	0   &   0   &   1   &   0   \\
	0   &   0   &   0   &   1
\end{array}\right).$$
$$b^d=\left(\begin{array}{cccc}
	2   &   0   &   0   &   0   \\
	0   &   0   &   0   &   0   \\
	0   &   0   &   0   &   0   \\
	0   &   0   &   0   &   \frac{1}{2}
\end{array}\right),
b^{\pi}=\left(\begin{array}{cccc}
	0   &   0   &   0   &   0   \\
	0   &   1   &   0   &   0   \\
	0   &   0   &   1   &   0   \\
	0   &   0   &   0   &   0
\end{array}\right).$$
We compute that
$$a^dbaa^d=\mathrm{diag}(\frac{1}{2},0,0,0),$$ and so
$\| a^dbaa^d\|<1$. Moreover, we have $a^2ba^{\pi}=a^{\pi}aba, a^{\pi}b^2a=baba^{\pi}.$
But we have
$$aba^{\pi}=
\left(\begin{array}{cccc}
	0   &   0   &   0   &   0              \\
	0   &   0   &   0   &   0              \\
	0   &   0   &   0   &   2    \\
	0   &   0   &   0   &   0
\end{array}\right),
a^{\pi}ba=0,$$ and then $a^{\pi}ab\neq a^{\pi}ba$.

In this case,
$$a^db=\mathrm{diag}(\frac{1}{2},0,0,0), b^da^{\pi}=\mathrm{diag}(0,0,0,\frac{1}{2}),a^{\pi}baa^d=0, $$
$$aa^{\pi}=b^{\pi}aa^{\pi}=
\left(\begin{array}{cccc}
	0   &   0   &   0   &   0        \\
	0   &   0   &   0   &   0        \\
	0   &   0   &   0   &   1        \\
	0   &   0   &   0   &   0
\end{array}\right).$$
By the upper bound in Theorem \ref{thm2}, we should have
$$\|(a+b)^{\#}-a^d\|\leq 2+0+0+0+1+0+1+0=4.$$
Indeed, since
$$(a+b)^{\#}
=\left(\begin{array}{cccc}
	\frac{2}{3}   &   0   &   0   &   0      \\
	0         &   1   &   0   &   0      \\
	0         &   0   &   0   &   \frac{1}{4}      \\
	0         &   0   &   0   &   \frac{1}{2}
\end{array}\right),$$
we compute that $$\|(a+b)^{\#}-a^d\|=\frac{13}{12}<4.$$
\end{exam}

\end{document}